\documentclass{article}

\usepackage{amsthm}
\usepackage{amsmath, amssymb}
\usepackage{hyperref}
\usepackage{tikz}
\usetikzlibrary{matrix}
\usepackage[matrix, arrow, curve]{xy}

\newtheorem{Thm}{Theorem}[section]
\newtheorem{Def}[Thm]{Definition}
\newtheorem{Lem}[Thm]{Lemma}

\newtheorem{Kor}[Thm]{Corollary}
\newtheorem{Rem}[Thm]{Remark}

\title{The generalized Vaserstein symbol revisited}

\author{Tariq Syed\\
Mathematisches Institut\\
Heinrich-Heine-Universit{\"a}t D{\"u}sseldorf\\
Universit{\"a}tsstra{\ss}e 1\\
40225 D{\"u}sseldorf, Germany\\
tariq.syed@gmx.de}

\date{\today} %\date{}

\begin{document}

\maketitle

\begin{abstract}
We give a construction of a generalized Vaserstein symbol associated to any finitely generated projective module of rank $2$ over a commutative ring with unit.
\end{abstract}

\tableofcontents

\section{Introduction}

In this paper, we revisit the construction of the generalized Vaserstein symbol given in \cite{Sy1}. The Vaserstein symbol map was originally introduced by Andrei Suslin and Leonid Vaserstein in \cite[\S 5]{SV} as follows: If $n \geq 1$ is an integer and $R$ is a commutative ring with unit, a unimodular row of length $n$ over $R$ is a row vector $(a_{1},...,a_{n})$ of length $n$ with $a_{i} \in R$, $1 \leq i \leq n$, such that $\langle a_{1},...,a_{n} \rangle = R$. Note that the group $GL_{n} (R)$ of invertible $n \times n$-matrices over $R$ (and hence any of its subgroups) acts on the right on the set $Um_n (R)$ of unimodular rows of length $n$ over $R$ by matrix multiplication. Now if $a=(a_{1},a_{2},a_{3}) \in Um_{3} (R)$ is a unimodular row of length $3$ over $R$, then by definition there exists $b = (b_{1},b_{2},b_{3}) \in Um_3 (R)$ such that $a_{1}b_{1}+a_{2}b_{2}+a_{3}b_{3} = 1$. The invertible alternating matrix

\begin{center}
$V (a,b) = \begin{pmatrix}
0 & - a_1 & - a_2 & - a_3 \\
a_1 & 0 & - b_3 & b_2 \\
a_2 & b_3 & 0 & - b_1 \\
a_3 & - b_2 & b_1 & 0
\end{pmatrix}$
\end{center}

has Pfaffian $1$ and hence represents an element of the elementary symplectic Witt group $W_E (R)$ introduced in \cite[\S 3]{SV} which is independent of the choice of the elements $b_{1},b_{2},b_{3} \in R$. It is shown in \cite[\S 5]{SV} that the assignment $a \mapsto V(a,b)$ induces a well-defined map
\begin{center}
$V: Um_{3} (R)/E_{3}(R) \rightarrow W_E (R)$
\end{center}
called the Vaserstein symbol. Here $E_3 (R)$ denotes the subgroup of $GL_3 (R)$ generated by elementary matrices and $Um_{3} (R)/E_{3}(R)$ denotes the orbit space of $Um_3 (R)$ under the right action of $E_{3}(R)$. One also obtains a well-defined map
\begin{center}
$V: Um_{3} (R)/SL_{3}(R) \rightarrow W_{SL} (R)$,
\end{center}
where $W_{SL} (R)$ is an abelian group also defined in \cite[\S 3]{SV}. It is a quotient group of the abelian group $W_E (R)$. Again, $Um_{3} (R)/SL_{3}(R)$ denotes the orbit space of $Um_3 (R)$ under the right action of the group $SL_{3}(R)$.\\
In \cite{Sy1} the author gave a generalized construction of the Vaserstein symbol map: Given a rank $2$ projective $R$-module $P_{0}$ with a fixed trivialization $\theta_{0}: R \xrightarrow{\cong} \det (P_{0})$ of its determinant, he could construct a well-defined map
\begin{center}
$V_{\theta_{0}}: Um (P_{0} \oplus R)/E (P_{0} \oplus R) \rightarrow \tilde{V}(R)$
\end{center}
called the generalized Vaserstein symbol associated to $P_{0}$ and $\theta_{0}$. The abelian group $\tilde{V}(R)$ is a subgroup of the group $V (R)$ considered in \cite[Section 3.B]{Sy1} and is canonically isomorphic to the group $W_E (R)$. Furthermore, $Um (P_{0} \oplus R)$ denotes the set of $R$-linear epimorphisms $P_{0} \oplus R \rightarrow R$ and $E (P_{0} \oplus R)$ is the subgroup of the group $Aut (P_{0} \oplus R)$ of $R$-linear automorphisms of $P_{0} \oplus R$ generated by elementary automorphisms of $P_{0} \oplus R$. As a matter of fact, the set $Um (P_{0} \oplus R)$ can be identified with $Um_{3}(R)$ in case $P_{0} = R^2$ is free, while the groups $Aut (P_{0} \oplus R)$ and $E (P_{0} \oplus R)$ can be identified with the groups $GL_3 (R)$ and $E_3 (R)$ respectively in this case. By choosing an appropriate trivialization $\theta_{0} : R \xrightarrow{\cong} \det (R^2)$, one recovers by means of all these identifications the original Vaserstein symbol as defined by Suslin and Vaserstein in \cite[\S 5]{SV}. We refer the reader to \cite[Section 4.B]{Sy1} for details. It was also proven in \cite[Theorem 3.1]{Sy2} that the generalized Vaserstein symbol descends to a well-defined map
\begin{center}
$V_{\theta_{0}}: Um (P_{0} \oplus R)/SL (P_{0} \oplus R) \rightarrow \tilde{V}_{SL}(R)$.
\end{center}
Here $SL (P_{0} \oplus R)$ denotes the group of $R$-linear automorphisms of $P_{0} \oplus R$ with determinant $1$ and the abelian group $\tilde{V}_{SL}(R)$ is a subgroup of the abelian group $V_{SL} (R)$ considered in \cite[Section 2.C]{Sy2} and is canonically isomorphic to the group $W_{SL}(R)$; in fact, $\tilde{V}_{SL}(R)$ is a quotient of $\tilde{V} (R)$. The orbit spaces $Um (P_{0} \oplus R)/Aut (P_{0} \oplus R)$, $Um (P_{0} \oplus R)/SL (P_{0} \oplus R)$ and $Um (P_{0} \oplus R)/E (P_{0} \oplus R)$ play a central role in the study of the cancellation property of the projective $R$-module $P_{0}$; we refer the reader to \cite[Section 2.D]{Sy1} for a detailed explanation of this viewpoint.\\
In this paper, we extend the construction of the generalized Vaserstein symbol to all rank $2$ projective $R$-modules: First of all, we define abelian groups $V (R, L)$ and $V_{SL} (R, L)$ for any rank $1$ projective $R$-module $L$; when $R = L$, these groups can be identified with the group $V (R)$ considered in \cite[Section 3.B]{Sy1} and with the group $V_{SL} (R)$ considered in \cite[Section 2.C]{Sy2} respectively. We also define a quotient $V_{GL}(R,L)$ of $V_{SL}(R,L)$ which is isomorphic to $\tilde{V}_{SL}(R)$ when $R = L$. Then we construct a well-defined map
\begin{center}
$V: Um (P_{0} \oplus R)/E (P_{0} \oplus R) \rightarrow V(R, \det (P_{0}))$,
\end{center}
which we call the generalized Vaserstein symbol of $P_{0}$ (cf. Theorems \ref{T3.2} and \ref{Maps}). We also show in Theorem \ref{Maps} that this map descends to a well-defined map
\begin{center}
$V: Um (P_{0} \oplus R)/SL (P_{0} \oplus R) \rightarrow V_{SL}(R, \det (P_{0}))$
\end{center}
and therefore also to a well-defined map
\begin{center}
$V: Um (P_{0} \oplus R)/SL (P_{0} \oplus R) \rightarrow V_{GL}(R, \det (P_{0}))$.
\end{center}
If the rank $2$ projective $R$-module $P_{0}$ admits a trivialization $\theta_{0} : R \xrightarrow{\cong} \det (P_{0})$ of its determinant, we recover the definition of the generalized Vaserstein symbol given in \cite{Sy1} by means of the identification $V (R) \cong V (R, \det (P_{0}))$ induced by $\theta_{0}$; similarly, we recover the definition of the induced maps $V_{\theta_{0}}: Um (P_{0} \oplus R)/SL (P_{0} \oplus R) \rightarrow \tilde{V}_{SL}(R) \subset {V}_{SL}(R)$ by means of the identifications $V_{SL} (R) \cong V_{SL} (R, \det (P_{0}))$ or $\tilde{V}_{SL} (R) \cong V_{GL} (R, \det (P_{0}))$ induced by $\theta_{0}$.\\
The original Vaserstein symbol as defined in \cite{SV} highlights a remarkable connection between stably free modules, Hermitian $K$-theory and homotopy theory as it induces the explicit KO-degree map $\Psi_3$ in Hermitian $K$-theory considered in \cite{AF} which stabilizes to the unit map from the unit sphere spectrum to the Hermitian $K$-theory spectrum in motivic homotopy theory; as a matter of fact, the original Vaserstein symbol was substantially used in the proof of breakthrough results on stably free modules (cf. \cite{F}, \cite{FRS}). Other cancellation results on projective modules were obtained in \cite{Sy3} by using the generalized Vaserstein symbol defined in \cite{Sy1}. Finally, the generalized Vaserstein symbol as defined in \cite{Sy1} has led to a conceptual understanding of the cancellation property of rank $2$ projective modules with a trivial determinant over Noetherian rings of dimension $\leq 4$ (cf. \cite{Sy2}); this was considered a particularly difficult and subtle problem as examples of smooth affine algebras of dimension $4$ over algebraically closed fields which admit non-trivial stably trivial modules of rank $2$ were constructed in \cite{NMK}. It should be expected that further analysis of the generalized Vaserstein symbol defined in this paper will lead to an analogous understanding of the cancellation property of rank $2$ projective modules (whose determinant is not necessarily trivial) over Noetherian rings of dimension $\leq 4$.\\
The paper is organized as follows: In Section \ref{2} we introduce the abelian groups $V (R,L)$, $V_{SL}(R,L)$ and $V_{GL}(R,L)$ of a commutative ring $R$ and a rank $1$ projective $R$-module $L$. In Section \ref{3} we give the construction of the generalized Vaserstein symbol of a finitely generated projective $R$-module $P_{0}$ of rank $2$ and prove that the generalized Vaserstein symbol as well as all the induced maps are well-defined.

\subsection*{Acknowledgements}

The author would like to thank the anonymous referee for suggesting changes which greatly improved the exposition of the paper. Furthermore, the author would like to thank his doctoral advisors Jean Fasel and Andreas Rosenschon once again for all their support during his doctorate on the generalized Vaserstein symbol a few years ago. The author was funded by the Deutsche Forschungsgemeinschaft (DFG, German Research Foundation) - Project number 544731044.

\section{Preliminaries}\label{Preliminaries}\label{2}

Throughout this paper, $R$ will denote a commutative ring with unit. Furthermore, let $P$ be a finitely generated projective $R$-module.\\
We denote by $Um (P)$ the set of $R$-linear epimorphisms $P \rightarrow R$. Furthermore, we denote by $Aut (P)$ the group of $R$-linear automorphisms of $P$ and we let $SL(P)$ denote its subgroup of $R$-linear automorphisms with determinant $1$. Note that $Aut (P)$ (and hence any of its subgroups) acts on the right on $Um (P)$ by precomposition.\\
Following \cite[Section 2.B]{Sy1}, we define elementary automorphisms of a direct sum $P = \bigoplus_{i=1}^{n} P_i$. An $R$-linear homomorphism $s_{ij}: P_j \rightarrow P_i$ for some $i \neq j$ is always understood to be extended to an endomorphism of $P$ by defining it to be $0$ on the other direct summands. We call automorphisms of $P$ of the form $\varphi_{s_{ij}} = id_{P} + s_{ij}$, where $s_{ij}$ is an $R$-linear homomorphism $s_{ij}: P_j \rightarrow P_i$ for some $i \neq j$, elementary automorphisms with respect to the given decomposition into a direct sum of the module $P$; if the decomposition of $P$ into a direct sum is understood, we just call them elementary automorphisms. The subgroup of $Aut (P)$ generated by elementary automorphisms will be denoted by $E(P_{1},...,P_{n})$ or simply by $E (P)$ if the decomposition of  $P$ into a direct sum is understood.\\
For any integer $n \geq 1$, we denote by $Um_n (R)$ the set of unimodular rows of length $n$ over $R$, i.e., the set of row vectors $(a_{1},...,a_{n})$ of length $n$ with $a_{i} \in R$, $1 \leq i \leq n$, such that $\langle a_{1},...,a_{n} \rangle = R$; if $P = R^n$ for some integer $n \geq 1$, we naturally identify the sets $Um_{n}(R)$ and $Um (R^n)$. In this case we also identify $Aut (P)$ with the group  $GL_n (R)$ of invertible $n \times n$-matrices over $R$ and $SL (P)$ with the group $SL_n (R)$ of invertible $n \times n$-matrices over $R$ with determinant $1$; it follows from \cite[Lemma 2.7(a)]{SV} that the group $E (P \oplus R)$ can be identifed with the group $E_{n+1}(R)$ generated as a subgroup of $GL_{n+1}(R)$ by elementary matrices. Note that if $P = R^n$ the action of $Aut (P)$ on $Um (P)$ given by precomposition then corresponds to the action of $GL_n (R)$ on $Um_n (R)$ given by matrix multiplication.\\
For any rank $1$ projective $R$-module $L$, we set $P^{\vee_{L}} = Hom_{R-\mathfrak{mod}} (P, L)$. One has a natural isomorphism

\begin{center}
$can_{L}: P \rightarrow P^{\vee_{L} \vee_{L}}, p \mapsto (ev_{p}:P^{\vee_{L}} \rightarrow L, a \mapsto a (p))$,
\end{center}

induced by evaluation.
\begin{Def}
An $L$-oriented alternating homomorphism on $P$ is an $R$-linear homomorphism $f: P \rightarrow P^{\vee_{L}}$ such that $f(p)(p) = 0$ for all $p \in P$. An $L$-oriented alternating isomorphism on $P$ is an $L$-oriented alternating homomorphism on $P$ which is an isomorphism.
\end{Def}
Note that if $L = R$, then $L$-oriented alternating homomorphisms just correspond to alternating $R$-bilinear forms $P \times P \rightarrow R$ and $L$-oriented alternating isomorphisms to non-degenerate alternating $R$-bilinear forms $P \times P \rightarrow R$ (cf. \cite[Section 2.A]{Sy1}). For any finitely generated projective $R$-module $P$, there is an $L$-oriented alternating isomorphism
\begin{center}
$H_{L} (P): P \oplus P^{\vee_{L}} \rightarrow P^{\vee_{L}} \oplus P^{\vee_{L} \vee_{L}}$
\end{center}
given by
\begin{center}
$
\begin{pmatrix}
0 & id \\
- can_{L} & 0
\end{pmatrix}
$,
\end{center}

which we call the hyperbolic $L$-oriented alternating isomorphism.\\
We say that two $L$-oriented alternating homomorphisms $f: P \rightarrow P^{\vee_{L}}$ and $g: Q \rightarrow Q^{\vee_{L}}$ on finitely generated projective $R$-modules $P$ and $Q$ are isometric if there is an $R$-linear isomorphism $\varphi: P \rightarrow Q$ such that $\varphi^{\vee_{L}} g \varphi = f$.\\
We now consider triples $(P, g, f)$, where $P$ is a finitely generated projective $R$-module and $g$ and $f$ are $L$-oriented alternating isomorphisms on $P$. We say that two such triples $(P,g,f)$ and $(P',g',f')$ are isometric if there is an $R$-linear isomorphism $\varphi: P \rightarrow P'$ such that $\varphi^{\vee_{L}} f' \varphi = f$ and $\varphi^{\vee_{L}} g' \varphi = g$. We denote the isometry class of a triple $(P,g,f)$ simply by $[P,g,f]$.
\begin{Def}
The abelian group $V (R, L)$ is defined as the quotient of the free abelian group on the set of isometry classes of triples $(P,g,f)$ as above by the subgroup generated by the relations
\begin{center}
\begin{itemize}
\item $[P \oplus P', g \perp g', f \perp f'] = [P,g,f] + [P',g',f']$ for any $L$-oriented alternating isomorphisms $f,g$ on $P$ and $f',g'$ on $P'$;
\item $[P, h, f] = [P, g, f] + [P, h, g]$ for any $L$-oriented alternating isomorphisms $f,g,h$ on $P$.
\end{itemize}
\end{center}
\end{Def}

If $R = L$, we recover the definition of the abelian group $V(R)$ discussed in \cite[Section 3.B]{Sy1}, i.e., $V (R) = V (R, R)$. Note that the defining relations of the group $V (R, L)$ above immediately imply that also the following relations hold in the group $V (R, L)$:
 
\begin{itemize}
\item $[P, f, f] = 0$ in $V (R, L)$ for any $L$-oriented alternating isomorphism $f$ on $P$,
\item $[P, g, f] = - [P, f, g]$ in $V (R, L)$ for any $L$-oriented alternating isomorphisms $f,g$ on $P$,
\item $[P, g, {\beta}^{\vee_{L}} {\alpha}^{\vee_{L}} f \alpha \beta] = [P, f, {\alpha}^{\vee_{L}} f \alpha] + [P, g, {\beta}^{\vee_{L}} f \beta]$ in $V (R, L)$ for all $\alpha, \beta \in Aut (P)$ and $L$-oriented alternating isomorphisms $f,g$ on $P$.
\end{itemize}

Similarly, we define the abelian group $V_{SL} (R, L)$:

\begin{Def}
The abelian group $V_{SL} (R, L)$ is defined as the quotient of the free abelian group on the set of isometry classes of triples $(P,g,f)$ as above by the subgroup generated by the relations
\begin{center}
\begin{itemize}
\item $[P \oplus P', g \perp g', f \perp f'] = [P,g,f] + [P',g',f']$ for any $L$-oriented alternating isomorphisms $f,g$ on $P$ and $f',g'$ on $P'$;
\item $[P, h, f] = [P, g, f] + [P, h, g]$ for any $L$-oriented alternating isomorphisms $f,g,h$ on $P$;
\item $[P, g, f] = [P, g, {\varphi}^{\vee_{L}} f \varphi]$ for any $L$-oriented alternating isomorphisms $f,g$ on $P$ and $\varphi \in SL (P)$.
\end{itemize}
\end{center}
\end{Def}

Clearly, $V_{SL} (R, L)$ is a quotient of $V (R, L)$. If $R = L$, we recover the definition of the abelian group $V_{SL} (R)$ discussed in \cite[2.C]{Sy2}, i.e., $V_{SL} (R) = V_{SL} (R, R)$. Finally, we define the abelian group $V_{GL} (R, L)$:

\begin{Def}
The abelian group $V_{GL} (R, L)$ is defined as the quotient of the free abelian group on the set of isometry classes of triples $(P,g,f)$ as above by the subgroup generated by the relations
\begin{center}
\begin{itemize}
\item $[P \oplus P', g \perp g', f \perp f'] = [P,g,f] + [P',g',f']$ for any $L$-oriented alternating isomorphisms $f,g$ on $P$ and $f',g'$ on $P'$;
\item $[P, h, f] = [P, g, f] + [P, h, g]$ for any $L$-oriented alternating isomorphisms $f,g,h$ on $P$;
\item $[P, g, f] = [P, g, {\varphi}^{\vee_{L}} f \varphi]$ for any $L$-oriented alternating isomorphisms $f,g$ on $P$ and $\varphi \in Aut (P)$.
\end{itemize}
\end{center}
\end{Def}

Again, $V_{GL} (R, L)$ is clearly a quotient of $V (R, L)$ and as well of $V_{SL} (R, L)$. If $R = L$, then $V_{GL}(R,L)$ can be identified with the group $\tilde{V}_{SL}(R)$ considered in \cite[Section 2.C]{Sy2}: Indeed, the canonical isomorphism $V(R,R) = V (R) \cong W'_E (R)$ from \cite[Section 3.B]{Sy1} descends to a canonical isomorphism between $V_{GL}(R,R)$ and $coker (K_{1}(R) \xrightarrow{H} W'_E (R))$, where $H$ denotes the hyperbolic map considered in \cite[Lemma 2.4]{F}; it is easy to see that $coker (K_{1}(R) \xrightarrow{H} W'_E (R)) \cong coker (SK_{1}(R) \xrightarrow{\tilde{H}} W_E (R)) = W_{SL} (R)$, where $\tilde{H}$ denotes the map induced by $H$ (cf. \cite[Proposition 2.7]{F}). Since it was shown in \cite[Section 2.C]{Sy2} that $W_{SL}(R) \cong \tilde{V}_{SL}(R)$, it follows that $V_{GL}(R, R)$ and $\tilde{V}_{SL}(R)$ are indeed canonically isomorphic. We now prove the following useful lemma:

\begin{Lem}\label{L2.1}
Let $P = \bigoplus_{i=1}^{n} P_{i}$ be a finitely generated projective $R$-module and $f$ be an $L$-oriented alternating isomorphism on $P$. Then $[P, f, {\varphi}^{\vee} f {\varphi}] = 0$ in $V (R, L)$ for any element ${\varphi} \in E (P)$ with respect to the given decomposition. In particular, if $g$ is another $L$-oriented alternating isomorphism on $P$, then $[P,g,f]=[P,g,{\varphi}^{\vee_{L}} f \varphi] \in V (R, L)$ for any element ${\varphi} \in E (P)$ with respect to the given decomposition.
\end{Lem}

\begin{proof}
First of all, note that in the abelian group $V (R, L)$ the equality
 
\begin{center}
$[P, f, {\varphi}_{2}^{\vee_{L}} {\varphi}_{1}^{\vee_{L}} f {\varphi}_{1} {\varphi}_{2}] = [P, f, {\varphi}_{1}^{\vee_{L}} f {\varphi}_{1}] + [P, f, {\varphi}_{2}^{\vee_{L}} f {\varphi}_{2}]$
\end{center}

holds for any $\varphi_{1}, \varphi_{2} \in Aut (P)$. As a direct consequence, we only have to prove the lemma for elementary automorphisms.\\
So let $\varphi_{s}$ be the elementary automorphism induced by $s: P_{j} \rightarrow P_{i}$ for some $i \neq j$. Since we can always add the summand $[P, f, f] = 0$, we may assume that we are in the situation of \cite[Corollary 2.5]{Sy1}. Therefore we may assume that $\varphi_{s}$ is a commutator in $Aut (P)$. So write $\varphi_{s} =  {\varphi}_{1} {\varphi}_{2} {\varphi}_{1}^{-1} {\varphi}_{2}^{-1}$ for some $\varphi_{1}, \varphi_{2} \in Aut (P)$. Then we have 
\begin{center}
$[P, f, {\varphi}^{\vee_{L}}_{s} f {\varphi_{s}}] = [P, f, {\varphi}_{1}^{\vee_{L}} f {\varphi}_{1}] + [P, f, {\varphi}_{2}^{\vee_{L}} f {\varphi}_{2}] + [P, f, {({\varphi}_{1}^{-1})}^{\vee_{L}} f {\varphi}_{1}^{-1}]+ [P, f, {({\varphi}_{2}^{-1})}^{\vee_{L}} f {\varphi}_{2}^{-1}]$.
\end{center}

But since
\begin{center}
$[P, f, {\varphi}_{i}^{\vee_{L}} f {\varphi}_{i}] + [P, f, ({{\varphi}_{i}^{-1})}^{\vee_{L}} f {\varphi}_{i}^{-1}] = [P, f, \varphi_{i}^{\vee_{L}}({{\varphi}_{i}^{-1})}^{\vee_{L}} f {\varphi}_{i}^{-1}\varphi_{i}] = [P,f,f] = 0$
\end{center}
for $i=1,2$, it follows that indeed $[P, f, {\varphi}^{\vee_{L}}_{s} f {\varphi}_{s}]=0$. This proves the first statement.\\
The second statement easily follows from the first statement by using the defining relations in the group $V (R, L)$: Indeed, we have
\begin{center}
$[P,g,{\varphi}^{\vee_{L}} f \varphi] = [P,f,{\varphi}^{\vee_{L}} f \varphi] + [P,g,f] = [P,g,f]$
\end{center}
in $V (R, L)$. This finishes the proof.
\end{proof}

\begin{Rem}
Let $R$ be a smooth affine algebra over a perfect field $k$ with $char(k) \neq 2$. Then the group $W'_E (R)$ and hence $V (R) = V (R,R)$ can be identified with the higher Grothendieck-Witt group $GW_{1}^{3}(R)$ (cf. \cite[Section 2.3.2]{AF}); moreover, the group $V_{SL} (R)=V_{SL} (R,R)$ can be identified with the cokernel of the restricted hyperbolic map ${H_{1,3}}|_{SK_{1}(R)}: SK_{1}(R) \rightarrow GW_{1}^{3}(R)$ and the group $V_{GL} (R) \cong \tilde{V}_{SL} (R)$ can be identified with the cokernel of the hyperbolic map $H_{1,3}: K_{1}(R) \rightarrow GW_{1}^{3}(R)$. Now recall that one can also define higher Grothendieck-Witt groups $GW_{1}^{3}(R, L)$ for any rank $1$ projective $R$-module $L$ (cf. \cite[Definition 2.1.1]{AF}). It should be expected that the groups $V (R, L)$,  $V_{SL} (R, L)$ and $V_{GL} (R, L)$ can be identified with the higher Grothendieck-Witt group $GW_{1}^{3}(R, L)$, with the cokernel of the restricted hyperbolic map ${H_{1,3}}|_{SK_{1}(R)}: SK_{1}(R) \rightarrow GW_{1}^{3}(R, L)$ and with the cokernel of the hyperbolic map $H_{1,3}: K_{1}(R) \rightarrow GW_{1}^{3}(R, L)$ respectively.
\end{Rem}

\section{Results}\label{Results}\label{3}

Let $P_0$ be a finitely generated projective $R$-module of rank $2$ and $L = \det (P_{0})$ be its determinant. We denote by $\chi_0$ the $L$-oriented alternating isomorphism on $P_0$ given by $\chi_0 : P_0 \rightarrow P_0^{\vee_{L}}, q \mapsto (P_{0} \rightarrow L, p \mapsto p \wedge q)$.\\
As usual, let $Um (P_0 \oplus R)$ be the set of $R$-linear epimorphisms $P_0 \oplus R \rightarrow R$. Any element $a$ of $Um (P_0 \oplus R)$ induces an exact sequence of the form

\begin{center}
$0 \rightarrow P(a) \rightarrow P_0 \oplus R \xrightarrow{a} R \rightarrow 0$,
\end{center}

\noindent where we let $P(a) = \ker (a)$. The choice of a section $s: R \rightarrow P_0 \oplus R$ of $a$ then determines a retraction $r: P_0 \oplus R \rightarrow P(a)$ given by $r(p)= p - s a(p)$ and also an isomorphism $i_{s}: P_0 \oplus R \rightarrow P(a) \oplus R, p \mapsto a(p) + r(p)$. Then the isomorphism $i_{s}: P_{0} \oplus R \xrightarrow{\cong} P(a) \oplus R$ induces an isomorphism $\theta: \det(P_0) \xrightarrow{\cong} \det (P(a))$. Finally, we denote by $\chi_a$ the $L$-oriented alternating isomorphism on $P(a)$ given by $P(a) \rightarrow {P(a)}^{\vee_{L}}, q \mapsto (P(a) \rightarrow L, p \mapsto \theta^{-1} (p \wedge q))$.\\
We now want to define the generalized Vaserstein symbol
 
\begin{center}
$V: Um (P_0 \oplus R) \rightarrow V (R, L)$
\end{center}
 
associated to $P_{0}$ by

\begin{center}
$V (a) = [P_{0} \oplus R \oplus R^{\vee_{L}}, \chi_{0} \perp H_{L} (R), {(i_s \oplus 1)}^{\vee_{L}} (\chi_{a} \perp H_{L} (R)) {(i_s \oplus 1)}]$.
\end{center}

\begin{Rem}\label{R1}
Note that if there is an isomorphism $\theta_{0}: R \xrightarrow{\cong} \det (P_{0})$, this isomorphism induces a group isomorphism $V (R) \cong V (R, L)$ as $L$-oriented alternating isomorphisms will then correspond to non-degenerate alternating forms. As a matter of fact, the $L$-oriented alternating isomorphisms $\chi_{0} \perp H_{L} (R)$ and ${(i_s \oplus 1)}^{\vee_{L}} (\chi_{a} \perp H_{L} (R)) {(i_s \oplus 1)}$ considered above can then be identified with the non-degenerate alternating forms $\chi_{0} \perp \psi_{2}$ and ${(i_s \oplus 1)}^{t} (\chi_{a} \perp \psi_{2}) {(i_s \oplus 1)}$ considered in the definition of the generalized Vaserstein symbol given in \cite[Section 4.B]{Sy1}.
\end{Rem}

In order to prove that this generalized Vaserstein symbol is well-defined, we show that our definition is independent of the choice of a section of $a$:

\begin{Thm}\label{T3.2}
The generalized Vaserstein symbol is well-defined, i.e., the element $V (a)$ of $V (R, L)$ defined as above is independent of the choice of a section of $a$.
\end{Thm}

\begin{proof}
Our proof is a generalization of the proof of \cite[Theorem 4.1]{Sy1}. We let $a: P_{0} \oplus R \rightarrow R$ be an $R$-linear epimorphism and $s, t: R \rightarrow P_0 \oplus R$ two chosen sections of $a$. As in the definition above, $i_s$ and $i_t$ denote the isomorphisms $P_0 \oplus R \cong P (a) \oplus R$ induced by the two sections $s$ and $t$ respectively. Note that the isomorphism $\det (P(a)) \cong \det (P_0)$ does not depend on the choice of a section because the difference of two sections maps $R$ into $P(a)$; as a direct consequence, the form $\chi_{a}$ is independent of the choice of a section as well.\\
The theorem is therefore proven as soon as we show that the elements 
\begin{center}
$V_{a,s} = [P_{0} \oplus R \oplus R^{\vee_{L}}, \chi_{0} \perp H_{L} (R), {(i_s \oplus 1)}^{\vee_{L}} (\chi_{a} \perp H_{L} (R)) {(i_s \oplus 1)}] \in V (R, L)$
\end{center}
and
\begin{center}
$V_{a,t} = [P_{0} \oplus R \oplus R^{\vee_{L}}, \chi_{0} \perp H_{L} (R), {(i_t \oplus 1)}^{\vee_{L}} (\chi_{a} \perp H_{L} (R)) {(i_t \oplus 1)}] \in V (R, L)$
\end{center}
are in fact equal. Analogously to the proof of \cite[Theorem 4.1]{Sy1}, we will show this in the following three steps:
 
\begin{itemize}
\item We define an $R$-linear homomorphism $d_{L} : P_0 \oplus R \rightarrow R^{\vee_{L}}$. This homomorphism induces an automorphism $\varphi_L \in E (P_{0} \oplus R \oplus R^{\vee_{L}})$ defined by $\varphi_{L} = id_{P_0 \oplus R \oplus R^{\vee_{L}}} - d_{L}$;
\item We show that ${\varphi}_{L}^{\vee_{L}} {(i_s \oplus 1)}^{\vee_{L}} (\chi_{a} \perp H_{L} (R)) {(i_s \oplus 1)} {\varphi}_{L} = {(i_t \oplus 1)}^{\vee_{L}} (\chi_{a} \perp H_{L} (R)) {(i_t \oplus 1)}$;
\item Lemma \ref{L2.1} then implies that $V_{a, s} = V_{a, t}$.
\end{itemize}
 
For the first step, we first define an $R$-linear homomorphism $d'_{L}: P_0 \oplus R \rightarrow \det (P_0 \oplus R)$ by $p \mapsto s(1) \wedge t(1) \wedge p \in \det (P_0 \oplus R)$ and then let $d_{L}: P_0 \oplus R \rightarrow R^{\vee_{L}}$ be the $R$-linear homomorphism obtained from $d'_{L}$ by composing with the canonical isomorphism $\det (P_0 \oplus R) \cong \det (P_0) \cong R^{\vee_{L}}$. Finally, we let $\varphi_{L} = id_{P_0 \oplus R \oplus R^{\vee_{L}}} - d_{L}  \in E (P_{0} \oplus R \oplus R^{\vee_{L}})$ be the induced elementary automorphism.\\
For the second step, we can check the desired equality locally at every prime ideal of $R$. So let $\mathfrak{p}$ be a prime ideal of $R$ and let $\theta_{0}: R_{\mathfrak{p}} \xrightarrow{\cong} L_{\mathfrak{p}}$ be an $R_{\mathfrak{p}}$-linear isomorphism. Then, following Remark \ref{R1}, we can identify the localization at $\mathfrak{p}$ of the alternating isomorphism ${(i_s \oplus 1)}^{\vee_{L}} (\chi_{a} \perp H_{L} (R)) {(i_s \oplus 1)}$ with the non-degenerate alternating form ${(i_{s_{\mathfrak{p}}} \oplus 1)}^{t} (\chi_{a_{\mathfrak{p}}} \perp \psi_{2}) {(i_{s_{\mathfrak{p}}} \oplus 1)}$, where $a_{\mathfrak{p}}$ and $s_{\mathfrak{p}}$ are the localizations of $a$ and $s$ at $\mathfrak{p}$ respectively. Analogously, we can identify the localization at $\mathfrak{p}$ of the alternating isomorphism ${(i_t \oplus 1)}^{\vee_{L}} (\chi_{a} \perp H_{L} (R)) {(i_t \oplus 1)}$ with the non-degenerate alternating form ${(i_{t_{\mathfrak{p}}} \oplus 1)}^{t} (\chi_{a_{\mathfrak{p}}} \perp \psi_{2}) {(i_{t_{\mathfrak{p}}} \oplus 1)}$. Furthermore, the localization at $\mathfrak{p}$ of $\varphi_{L}$ can be identifed with the automorphism $\varphi$ (now defined for $a_{\mathfrak{p}}$, $s_{\mathfrak{p}}$ and $t_{\mathfrak{p}}$) considered in the first step of the proof of \cite[Theorem 4.1]{Sy1}. With all these canonical identifications, it follows immediately from the second step in the proof of \cite[Theorem 4.1]{Sy1} that the desired equality
\begin{center}
${\varphi}_{L}^{\vee_{L}} {(i_s \oplus 1)}^{\vee_{L}} (\chi_{a} \perp H_{L} (R)) {(i_s \oplus 1)} {\varphi}_{L} = {(i_t \oplus 1)}^{\vee_{L}} (\chi_{a} \perp H_{L} (R)) {(i_t \oplus 1)}$
\end{center}
holds after localization at every prime $\mathfrak{p}$ of $R$ and therefore also globally; in other words, the desired equality
\begin{center}
${\varphi}_{L}^{\vee_{L}} {(i_s \oplus 1)}^{\vee_{L}} (\chi_{a} \perp H_{L} (R)) {(i_s \oplus 1)} {\varphi}_{L} = {(i_t \oplus 1)}^{\vee_{L}} (\chi_{a} \perp H_{L} (R)) {(i_t \oplus 1)}$
\end{center}
holds.\\
For the third and final step, since $\varphi_L \in E (P_{0} \oplus R \oplus R^{\vee_{L}})$, the second statement in Lemma \ref{L2.1} immediately implies that indeed $V_{a,s} = V_{a,t} \in V (R, L)$, as desired.
\end{proof}

\begin{Rem}\label{R2}
It follows from Remark \ref{R1} that if there is an isomorphism $\theta_{0}: R \xrightarrow{\cong} \det (P_{0})$, then the generalized Vaserstein symbol associated to $P_{0}$ above coincides by means of the identification $V (R) \cong V (R, L)$ with the generalized Vaserstein symbol associated to $P_{0}$ together with the trivialization $\theta_{0}$ of its determinant defined in \cite[Section 4.B]{Sy1}.
\end{Rem}

\begin{Thm}\label{Maps}
Let $a \in Um (P_{0} \oplus R)$ and $\varphi \in Aut (P_0 \oplus R)$. Then we have $V (a) = V (a \varphi) \in V (R, L)$ if $\varphi \in E (P_0 \oplus R)$ and $V (a) = V (a \varphi) \in V_{SL} (R, L)$ if $\varphi \in SL (P_0 \oplus R)$. In particular, the generalized Vaserstein symbol induces well-defined maps $V: Um (P_0 \oplus R)/E (P_0 \oplus R) \rightarrow V (R, L)$ and, by abuse of notation, $V: Um (P_0 \oplus R)/SL (P_0 \oplus R) \rightarrow V_{SL} (R, L)$.
\end{Thm}
 
\begin{proof}
We let $\varphi$ be an automorphism of $P_0 \oplus R$ and $a \in Um (P_0 \oplus R)$ with a chosen section $s: R \rightarrow P_0 \oplus R$. Then ${\varphi}^{-1} s$ is automatically a section of $a \varphi \in Um (P_{0} \oplus R)$. Furthermore, we let $i_{s}: P_0 \oplus R \rightarrow P(a) \oplus R$ and $i_{{\varphi}^{-1}s}: P_0 \oplus R \rightarrow P(a \varphi) \oplus R$ be the isomorphisms induced by the sections $s$ and ${\varphi}^{-1} s$ respectively. It suffices to prove that

\begin{center}
${(\varphi \oplus 1)}^{\vee_{L}} {(i_{s} \oplus 1)}^{\vee_{L}} (\chi_a \perp H_{L} (R)) {(i_{s} \oplus 1)} {(\varphi \oplus 1)} =$\\
${(i_{{\varphi}^{-1}s} \oplus 1)}^{\vee_{L}} (\chi_{(a \varphi)} \perp H_{L} (R)) {(i_{{\varphi}^{-1}s} \oplus 1)}$
\end{center}

as the theorem then follows immediately from Lemma \ref{L2.1} in case $\varphi \in E (P_{0} \oplus R)$ and from the definition of $V_{SL} (R, L)$ in case $\varphi \in SL (P_{0} \oplus R)$.\\
So let us prove that the desired equality above holds indeed. After unwinding definitions, one realizes that
\begin{center}
$(i_{s} \oplus 1) (\varphi \oplus 1) = (\overline{\varphi} \oplus 1 \oplus 1) (i_{{\varphi}^{-1}s} \oplus 1)$,
\end{center}
where we denote by $\overline{\varphi}$ the isomorphism $P (a \varphi) \rightarrow P(a)$ induced by $\varphi$. As a direct consequence, it only remains to show that ${\overline{\varphi}}^{\vee_{L}} \chi_a \overline{\varphi} = \chi_{a \varphi}$.\\
For this pupose, let $p,q \in P (a \varphi)$; by definition, $\chi_{a \varphi}(q)$ sends $p$ to the image of $p \wedge q$ under the isomorphism $\det (P (a \varphi)) \cong \det (P_{0})$. This element can also be described as the image of $p \wedge q \wedge {\varphi}^{-1} s(1)$ under the canonical isomorphism $\det (P_0 \oplus R) \cong \det (P_{0})$. An analogous computation shows that ${\overline{\varphi}}^{\vee_{L}} \chi_a \overline{\varphi} (q)$ sends $p$ to the image of the element ${\varphi} (p) \wedge {\varphi} (q) \wedge s(1)$ under the canonical isomorphism $\det (P_0 \oplus R) \cong \det (P_{0})$. The theorem then follows from \cite[Lemma 2.11]{Sy1}.
\end{proof}

\begin{Kor}
The generalized Vaserstein symbol induces a well-defined map $V: Um (P_0 \oplus R)/SL (P_0 \oplus R) \rightarrow V_{GL} (R, L)$.
\end{Kor}

\begin{proof}
Follows immediately from Theorem \ref{Maps}.
\end{proof}

\end{document}